\documentclass[english,12pt]{amsart}

\usepackage[latin1]{inputenc} 									
\usepackage[T1]{fontenc}
\usepackage{hyperref}     
\usepackage{geometry}
\usepackage[english]{babel}  
\usepackage{amsmath,amsfonts,amssymb,amsthm}
\usepackage{mathrsfs}          
\usepackage{dsfont}
\usepackage[all]{xy}
\usepackage{graphicx}
\usepackage{float}
\usepackage{subfig}								
\usepackage{tikz,tkz-graph}

\newtheorem{thm}{Theorem}[section]

\newtheorem{propo}[thm]{Proposition}
\newtheorem{lem}[thm]{Lemma}
\newtheorem{cor}[thm]{Corollary}

\theoremstyle{definition}
\newtheorem{de}[thm]{Definition}

\theoremstyle{remark}
\newtheorem{rmk}[thm]{Remark}

\newenvironment{notation}{\par\vspace{5pt}\noindent\textbf{Notation.}}{\par\vspace{5pt}}

\newcommand{\CC}{\mathds{C}}
\newcommand{\RR}{\mathds{R}}

\newcommand{\ZZ}{\mathds{Z}}

\newcommand{\FF}{\mathds{F}}
\newcommand{\PP}{\mathds{P}}
\newcommand{\BB}{\mathds{B}}
\newcommand{\A}{\mathcal{A}}
\renewcommand{\P}{\mathcal{P}}
\newcommand{\D}{\mathcal{D}}
\newcommand{\C}{\mathcal{C}}
\newcommand{\M}{\mathcal{M}}

\newcommand{\R}{\mathcal{R}}
\renewcommand{\L}{\mathcal{L}}

\newcommand{\set}[1]{\left\{ #1 \right\}}
\renewcommand{\epsilon}{\varepsilon}

\renewcommand{\tilde}{\widetilde}

\DeclareMathOperator{\HH}{H}
\DeclareMathOperator{\aut}{Aut}
\DeclareMathOperator{\GL}{GL}

\DeclareMathOperator{\dist}{d}

\begin{document}

\title{Multiplicativity of the $\mathcal{I}$-invariant and topology of glued arrangements}
\author{Benoît Guerville-Ballé}
\address{
Department of Mathematics  \newline\indent  
Tokyo Gaugei University  \newline\indent  
4-1-1 Nukuikita-machi  \newline\indent  
Koganeishi, Tokyo, 184-8501  \newline\indent  
Japan  \newline\indent 
}
\email{benoit.guerville-balle@math.cnrs.fr}

\urladdr{www.benoit-guervilleballe.com}

\thanks{Supported by a JSPS Postdoctoral Fellowship} 


\begin{abstract}
	The invariant $\mathcal{I}(\A,\xi,\gamma)$ was first introduced by E.~Artal, V.~Florens and the author. Inspired by the idea of G.~Rybnikov, we obtain a multiplicativity theorem of this invariant under the gluing of two arrangements along a triangle. An application of this theorem is to prove that the extended Rybnikov arrangements form an ordered Zariski pair (\emph{i.e.} two arrangements with the same combinatorial information and different ordered topologies). Finally, we extend this method to a family of arrangements and thus we obtain a method to construct new examples of Zariski pairs.
\end{abstract}

\maketitle

\section*{Introduction}

An important question in the study of an algebraic curve $\mathscr{C}\subset\CC\PP^2$ is to understand the relation between the combinatorial information of a curve and its topology (\emph{i.e.} the topological type of the pair $(\CC\PP^2,\mathscr{C})$). The first results are due to O.~Zariski in~\cite{Zar_Problem,Zar_Topological}, where he proves that the topology is not determined by the combinatorial information. Indeed, he constructs two sextics with the same combinatorial data and such that the fundamental groups of their complements are not isomorphic.

A specific case of algebraic plane curves is line arrangements. They are curves of which all irreducible components are of degree one. P.~Orlik and L.~Solomon prove in ~\cite{OrlSol}, that the cohomology ring of an arrangement is determined by its combinatorial information. This suggests that, in the case of line arrangements, the combinatorics determines the topology. But in~\cite{Ryb_Preprint,Ryb_Article}, G.~Rybnikov explicitly constructs an example like Zariski's one, in the case of arrangements. In this way, E.~Artal proposes, in~\cite{Art_Couples}, to call Zariski pairs such examples (\emph{i.e.} two curves with the same combinatorial information and different topologies).

As far as we know, only two other examples of Zariski pairs of line arrangements are already known. The second one is due to E.~Artal, J.~Carmona, J.I.~Cogolludo and M.A.~Marco in~\cite{ArtCarCogMar_Topology}. Furthermore, this example is the only one which is formed by two complexified real arrangements (\emph{i.e.} arrangements where the lines are defined by real equation). The third known example is obtained by the author in~\cite{Gue_Pairs}. The topologies of this example were distinguished using the invariant $\mathcal{I}(\A,\xi,\gamma)$ (also called the $\mathcal{I}$-invariant). The last two examples are arithmetic Zariski pairs: arrangements with equations conjugated in a number field.

The $\mathcal{I}$-invariant was introduced in~\cite{ArtFloGue} (see also~\cite{Gue_Thesis}). It can be viewed as an adaptation, in the case of line arrangements, of the linking number of the link theory. Inspired by the idea of G.~Rybnikov developed in~\cite{Ryb_Preprint,Ryb_Article}, we prove, in this paper, a theorem of multiplicativity of the invariant $\mathcal{I}(\A,\xi,\gamma)$ under the gluing of two arrangements along a triangle (supporting $\gamma$). As an illustration of this result we show that the extended Rybnikov arrangements form an ordered Zariski pair. Then, we generalize this construction to a family of arrangements. This provides a method to construct new examples of non-arithmetic Zariski pairs.

In Section~\ref{Sec_Invariant}, we recall the construction of the invariant $\mathcal{I}(\A,\xi,\gamma)$, define the extended MacLane arrangements and use them to illustrate the definitions previously given. After having defined the notion of the gluing of two arrangements, the multiplicativity theorem is stated and proved in Section~\ref{Sec_Multiplicativity}. In the first part of Section~\ref{Sec_Application}, we define the extended Rybnikov arrangements from the extended MacLane arrangements studied in Section~\ref{Sec_Invariant}; in the second part, we use the multiplicativity theorem to prove that the extended Rybnikov arrangements form an ordered Zariski pair. To finish this section, we extend the method, developed for the extended Rybnikov arrangements, to the family of arrangements for which the $\mathcal{I}$-invariant is not real.

\newpage

\section{The $\mathcal{I}$-invariant}\label{Sec_Invariant}

In this first section, we give the notion of inner-cyclic combinatorics and of inner-cyclic arrangements. Then, we recall the construction of the $\mathcal{I}$-invariant developed in~\cite{ArtFloGue}. To finish, we construct the extended MacLane arrangements and use it to illustrate the previous notions.

\subsection{Inner-cyclic combinatorics}

\begin{de}\label{Def_Combinatorics}
	A \emph{combinatorics} is a couple $\C=(\L,\P)$, where $\L$ is a finite set and $\P$ a subset of the power set of $\L$, satisfying that:
	\begin{enumerate}
		\item For all $P\in \P$, $\# P\geq 2$;
		\item For any $L_1,L_2 \in\L$, $L_1\neq L_2$, $\exists ! P\in\P$ such that $L_1,L_2\in P$.
	\end{enumerate}
	The combinatorics is \emph{ordered}, if the set $\L$ is ordered.
\end{de}

\begin{notation}
	To simplify the notation, an element $\set{L_i,\cdots,L_j}$ in $\P$ is sometimes denoted by~$P_{i,\cdots,j}$.
\end{notation}

The \emph{incidence graph} $\Gamma_\C$ of a combinatorics $\C=(\L,\P)$ is a way to encode it into a graph. It is defined as a non-oriented bi-partite graph where the set of vertices $V(\C)$ is decomposed into two sets:
\begin{equation*}
	V_\P(\C)=\set{v_P \mid P\in\P} \text{  and  } V_\L(\C)=\set{v_L \mid L\in\L},
\end{equation*}
and an edge of $\Gamma_\C$ joins $v_L\in V_\L(\C)$ to $v_P\in V_\P(\C)$ if and only of $L\in P$.

A \emph{character} on a combinatorics $\C$ is an application $\xi$ from $\L$ to $\CC^*$ such that $\prod\limits_{L\in \L}\xi(L) =1$. It can be extended into an application $\xi_*$ on $V(\C)$ by associating to any $v_P\in V_\P(\C)$ the product~$\prod\limits_{L\ni P} \xi(v_L)$. 

\begin{de}\label{Def_InnerCyclicCombinatorics}
	A character $\xi$ on a combinatorics $\C$ is \emph{inner-cyclic} if there exists a non trivial cycle $\gamma\in\HH_1(\Gamma_\C)$ such that:
	\begin{equation*}
		\forall v\in V(\C),\ \dist(v,\gamma)\leq 1 \Longrightarrow \xi_*(v)=1,
	\end{equation*}
	where $\dist$ is the usual distance on a graph.
\end{de}

\begin{rmk}\label{Rmk_InnerCyclicCombinatorics}
	Definition~\ref{Def_InnerCyclicCombinatorics} of an inner-cyclic character previously given is equivalent to the following points:
	\begin{enumerate}
		\item For all $v_L\in\gamma$, $\xi(L)=1$,
		\item For all $v_P\in\gamma$, if $L\in P$ then $\xi(L)=1$,
		\item For all $P \in L$ such that $v_L\in\gamma$, $\prod\limits_{L_i\in P} \xi(L_i)=1$.
	\end{enumerate}
\end{rmk}

\subsection{Realisations and invariant}

The combinatorics of $\A$ is the data of the set of lines, the set of singular points of $\A$ and the relation between these two sets. It can be defined as the poset of all the intersections of the elements of $\A$, with respect to the reverse inclusion. Let $\C$ be a combinatorics, a complex line arrangement $\A=\set{L_1,\cdots,L_n}$ of $\CC\PP^2$ is a \emph{realisation} of $\C$ if its combinatorics agrees with $\C$. An \emph{ordered realisation} of an ordered combinatorics is defined accordingly. The incidence graph of the combinatorics of an arrangement $\A$ is denoted by $\Gamma_\A$.

\begin{de}
	Let $\A=\set{L_1,\cdots,L_n}$ and $\A'=\set{L'_1,\cdots,L'_n}$ be two ordered realisations of the same combinatorics. A homeomorphism $\phi$ of $\CC\PP^2$, such that $\phi(\A)=\A'$ preserves the ordered if $\phi(L_i)=L'_i$ for all $i\in\set{1,\cdots,n}$; it preserves the orientation if $\phi$ respects the global orientation of $\CC\PP^2$ and the local orientation around the lines (ie it sends meridians on meridians with respect of their orientations).
\end{de}

Let $\A$ be a realisation of a combinatorics $\C$. A character $\xi$ on $\C$ naturally defines a character (also denoted $\xi$) on the first homology group of the complement $E_\A=\CC\PP^2\setminus \A$, by:
\begin{equation*}
	\xi : \left\{
	\begin{array}{ccc}
		\HH_1(E_\A) & \longrightarrow & \CC^* \\
		m_i & \longmapsto & \xi(L_i)
	\end{array}
	\right. ,
\end{equation*}
where $m_i$ is the meridian associated with the line $L_i$.

\begin{de}
	An \emph{inner-cyclic arrangement} is the data of a triplet $(\A,\xi,\gamma)$, where $\A$ is an arrangement, $\xi$ an inner-cyclic character on the combinatorics $\C_\A$ of $\A$ and $\gamma\in\HH_1(\Gamma_{\A})$ the associated cycle. The support of $\gamma$ is the set $\set{L\in\A \mid v_L\in\gamma }$.
	If $\gamma$ is supported by 3 lines, then $(\A,\xi,\gamma)$ is a \emph{triangular} inner-cyclic arrangement.
\end{de}

\begin{notation}
	If an arrangement is triangular inner-cyclic, then we assume, in all the following, that the cycle $\gamma$ is supported by the three first lines of $\A$.
\end{notation}

Let $B_\A$ be the boundary manifold of an arrangement $\A$. It can be defined as the boundary of a regular neighbourhood of $\A$; let us remark that $B_\A\subset E_\A$. By~\cite{Wes}, it is a graph manifold over the incidence graph. Then, $B_\A$ can be decomposed into:
\begin{equation*}
	B_\A=\bigcup\limits_{L\in\L} \mathcal{N}_L \cup \bigcup\limits_{P\in\P}\BB_P,
\end{equation*}
where $\mathcal{N}_L$ is a $S^1$-bundle over $L\setminus \bigcup\limits_{P\in L} D_P$ (with $D_P$ an open disc of $L$ centered in $P$); and $\BB_P$ is the boundary of a 4-ball centered in $P$ without an open tubular neighbourhood of $\A$. There is projection $\rho$ from $\HH_1(B_\A)$ into $\HH_1(\Gamma_\A)$, well defined up to homotopy. 

Let $\gamma$ be a cycle of $\HH_1(\Gamma_\A)$. A \emph{nearby cycle} $\tilde{\gamma}$ associated with $\gamma$ is an embedded $S^1$ in $B_\A$ such that:
\begin{enumerate} 
	\item $\tilde{\gamma}\subset B_\A\setminus \Big( \big( \bigcup\limits_{v_P\notin \gamma} \BB_P \big)  \cup  \big( \bigcup\limits_{v_L\notin \gamma} \mathcal{N}_L \big)  \Big)$,
	\item $\rho([\tilde{\gamma}])=\gamma$, where $[\tilde{\gamma}]$ is the class of $\tilde{\gamma}$ in $\HH_1(B_\A)$. 
\end{enumerate}
Let $i_*:\HH_1(B_\A)\rightarrow\HH_1(E_\A)$ be the application induced by the inclusion of $B_\A$ in $E_\A$. If $(\A,\xi,\gamma)$ is an inner-cyclic arrangement. We define $\mathcal{I}(\A,\xi,\gamma)$ by:
\begin{equation*}
	\mathcal{I}(\A,\xi,\gamma)= \xi \circ i_* ([\tilde{\gamma}]),
\end{equation*}
where $\tilde{\gamma}$ is a nearby cycle in $B_\A$ associated with $\gamma$.
By~\cite[Lemma 2.2]{ArtFloGue}, $\mathcal{I}(\A,\xi,\gamma)$ does not depend of the choice of $\tilde{\gamma}$.

\begin{thm}[\cite{ArtFloGue}]\label{Thm_Invariance}
	Let $\A$ and $\A'$ be two ordered realisations of the same ordered combinatorics. If $(\A,\xi,\gamma)$ and $(\A',\xi,\gamma)$ are two inner-cyclic arrangements with the same oriented and ordered topological type then:
	\begin{equation*}
		\mathcal{I}(\A,\xi,\gamma)=\mathcal{I}(\A',\xi,\gamma).
	\end{equation*}
\end{thm}

\subsection{Extended MacLane arrangements}\label{SubSec_MacLane}

To illustrate the notions defined in these two previous subsections, and in prevision of the application of Theorem~\ref{Thm_Multiplicativity}, let us introduce the extended MacLane arrangements. These arrangements were first introduced in~\cite{ArtFloGue,Gue_Thesis}, as the first example of inner-cyclic arrangements distinguished by the invariant $\mathcal{I}(\A,\xi,\gamma)$. They can be defined as the usual MacLane arrangements, see~\cite{Mac,Whi}, with an additional line passing through two triple points.  

The combinatorics of the extended MacLane arrangement can be constructed as follows. Let $\PP\FF_3^2$ be the 2-dimension projective space on $\FF_3$, the fields of three elements, and consider the line $I~=~\set{[x:y:0]\mid x,y\in\FF_3}$ of $\PP\FF_3$ as the line at infinity. We define $\L_\M$ by the set $\PP\FF_3^2\setminus \set{[0:0:1],I} \cup Q$, where $Q$ is a point on $I$; and $\P_\M$ is constructed as follows: The elements of $\P_\M$ of cardinality greater (or equal) than 3 are the lines of $\PP\FF_3^2$ which do not pass through the point $[0:0:1]$ (for example $\set{1,2,3}\notin\P_\M$ but $\set{2,5,9}\in\P_\M$), and the elements of cardinality equal to 2 are such that the point~(2) of Definition~\ref{Def_Combinatorics} is verified.

This provides a combinatorics $\C_\M=(\L_\M,\P_\M)$ where the relation $\Subset$ between $\L_\M$ and $\C_\M$ is given by: for all $\ell\in\L_\M$, $P\in\P_\M$ such that $\# P \geq 3$, we have $P\Subset\ell \Leftrightarrow (\ell\in P \text{, in }\PP\FF_3^2)$, and we complete the relation $\Subset$ with the elements $P\in\P_\M$ such that $\# P = 2$. Figure~\ref{Fig_MacLaneCombinatorics} pictures the ordered extended MacLane combinatorics viewed in $\PP\FF_3^2$.

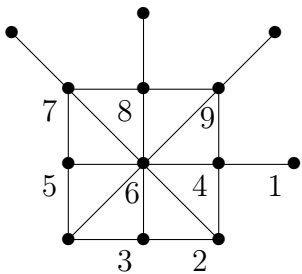
\begin{figure}[ht!]
	\centering
		\begin{tikzpicture}
			\begin{scope}
				\node (P2) at (1,0) {$\bullet$};
				\node (P3) at (2,0) {$\bullet$};
				\node (P5) at (0,1) {$\bullet$};
				\node (P0) at (1,1) {$\bullet$};
				\node (P6) at (2,1) {$\bullet$};
				\node (P4) at (0,2) {$\bullet$};
				\node (P1) at (1,2) {$\bullet$};
				\node (P7) at (2,2) {$\bullet$};
				
				\node (P8) at (1,3) {$\bullet$};
				\node (P9) at (2.75,2.75) {$\bullet$};
				\node (P10) at (3,1) {$\bullet$};
				\node (P11) at (-0.75,2.75) {$\bullet$};
				\node (P12) at (0,0) {$\bullet$};
				
				\node[below left] at (1,0) {3};
				\node[below left] at (2,0) {2};
				\node[below left] at (0,1) {5};
				\node[below left] at (1.1,0.9) {6};
				\node[below left] at (2,1) {4};
				\node[below left] at (0,2) {7};
				\node[below left] at (1,2) {8};
				\node[below left] at (2.1,1.9) {9};
				\node[below left] at (3,1) {1};
				
				\draw (1,0) -- (1,3);
				\draw (0,1) -- (3,1);
				\draw (0,0) -- (2.75,2.75);
				\draw (-0.75,2.75) -- (2,0);
				\draw (0,0) -- (0,2);
				\draw (2,0) -- (2,2);
				
				\draw (0,0) -- (2,0);
				\draw (0,2) -- (2,2);
				
			\end{scope}
		\end{tikzpicture}
		\caption{The ordered extended MacLane combinatorics viewed in $\PP\FF_3^2$ \label{Fig_MacLaneCombinatorics}}
\end{figure}

\begin{rmk}
	In order to obtain the MacLane arrangements, we delete the line $L_1$ of the extended MacLane arrangements, then double points are the lines passing through the ``origin''~$[0:0:1]$.
\end{rmk}

	With the notation of Definition~\ref{Def_Combinatorics}, we can define the extended MacLane combinatorics by $\L_\M~=~\set{L_1,\cdots,L_9}$ and 
	\begin{gather*}
		\P_\M = \left\{ 
		\set{L_1,L_2},
		\set{L_1,L_3},
		\set{L_1,L_4,L_5,L_6},
		\set{L_1,L_7,L_8,L_9}\right., \\
		\set{L_2,L_3},		
		\set{L_2,L_4,L_9}, 
		\set{L_2,L_5,L_8},
		\set{L_2,L_6,L_7},
		\set{L_3,L_4,L_7}, \\
		\left. 
		\set{L_3,L_5,L_9},
		\set{L_3,L_6,L_8},
		\set{L_4,L_8},
		\set{L_5,L_7},
		\set{L_6,L_9}\right\}.
	\end{gather*}
The line $L_1$ is the only one line of $\C_\M$ containing two points of multiplicity 4, thus it is fixed by all automorphisms of the combinatorics. This implies that the automorphism group of $\C_\M$ is a subgroup of  the one of the MacLane combinatorics, which is $\GL_2(\FF_3)$. Furthermore, the invariance of $L_1$ by automorphism implies that $L_2$ and $L_3$ are fixed or exchanged. The matrices realizing such condition are exactly $\left(\begin{array}{cc} a & 0 \\ b & c \end{array}\right)\in \GL_2(\FF_3)$, and all such matrices respect the combinatorics. Thus, we have that $\aut(\C_\M)\simeq \text{D}_6 \simeq \Sigma_3\times\ZZ_2$. Note that the $\ZZ_2$ part determines if $L_2$ and $L_3$ are fixed or exchanged; indeed this part corresponds to the value $1$ or $-1$ of the coefficient $a$ in the previous matrix.

As previously said, this combinatorics is inner-cylic. Let us consider the character $\xi_\M$ on $\C_\M$ defined by:
	\begin{equation*}
		\xi_\M : (L_1,\cdots,L_9) \longmapsto (1,1,1,\zeta,\zeta,\zeta,\zeta^2,\zeta^2,\zeta^2),
	\end{equation*}
	where $\zeta$ is a primitive 3-root of the unity. The character $\xi_\M$ is a triangular inner-cyclic character on $\C_\M$ for the cycle $\gamma_{(1,2,3)}\in\HH_1(\Gamma_{\C_\M})$ defined by:
	\begin{equation*}
	\begin{tikzpicture}[xscale=1.5]
		\node (A1) at (0,0) {$v_{L_1}$};
		\node (A2) at (1,0) {$v_{P_{1,2}}$};
		\node (A3) at (2,0) {$v_{L_2}$};
		\node (A4) at (3,0) {$v_{P_{2,3}}$};
		\node (A5) at (4,0) {$v_{L_{3}}$};
		\node (A6) at (5,0) {$v_{P_{1,3}}$};
		\draw[->] (A1)  -- (A2);
		\draw[->] (A2) -- (A3);
		\draw[->] (A3) -- (A4);
		\draw[->] (A4) -- (A5);
		\draw[->] (A5) -- (A6);
		\draw[->] (A6) to[out=110,in=0] (3.5,0.75) -- (1.5,0.75) to[out=180,in=80] (A1);
	\end{tikzpicture}
\end{equation*}

This combinatorics $\C_\M$ admits two complex realisations defined by:
\begin{gather*}
	L_1:z=0,\quad L_2:x-\bar{a} y=0,\quad L_3:x-a y=0,\quad L_4:y-\bar{a} z=0,\quad L_5:y-z=0,\\ 
	L_6:y-a z=0,\quad L_7:x-z=0,\quad L_8:x-\bar{a} z=0,\quad L_9:x-a z=0,
\end{gather*}
where $a=\zeta$ or $a=\zeta^2$ (with $\zeta$ a primitive cubic root of unity). These arrangements are denoted by $\M^+$ and $\M^-$, and are called the \emph{positive and negatively extended MacLane arrangements}. Let $\phi\in\aut(\C_\M)$ and let $\phi_*:\CC\PP^2\rightarrow\CC\PP^2$ be an application realizing $\phi$. If $\det(\phi)=-1$ then $\phi_*$ sent $\M^+$ on $\M^-$ and conversely; and if $\det(\phi)=1$ then $\phi_*$ fixes as a whole $\M^+$ and $\M^-$ , see~\cite{ArtCarCogMar_Invariants} for details.

The details of the computation of the $\mathcal{I}$-invariant for these arrangements are done in~\cite[Section 5]{ArtFloGue}. With the labelling of this article, we have that:
\begin{equation*}
	\mathcal{I}(\M^+,\xi_\M,\gamma_{(1,2,3)}) =\xi_\M\big( -m_7-m_9 \big) = \zeta^2
	\quad\text{ and }\quad
	\mathcal{I}(\M^-,\xi_\M,\gamma_{(1,2,3)}) =\xi_\M\big( -m_9 \big) = \zeta.
\end{equation*}

\section{Multiplicativity theorem}\label{Sec_Multiplicativity}

Inspired by the idea of G.~Rybnikov in~\cite{Ryb_Preprint}, we first explain how to glue two arrangements along a triangle. Then, we prove that such a gluing implies the multiplicativity of the invariant $\mathcal{I}(\A,\xi,\gamma)$.

Let $\A=\set{L_1,\cdots,L_n}$ and $\A'=\set{L'_1,\cdots,L'_k}$ be two ordered (by the indices) line arrangements such that $L_1$, $L_2$ and $L_3$ (resp. $L'_1$, $L'_2$ and $L'_3$) are in generic position (\emph{i.e.}~$L_1\cap L_2\cap L_3=\emptyset$). 

\begin{de}
	A \emph{gluing} of $\A$ and $\A'$ (in this order) is a projective transformation $\phi$ preserving the orientation and such that:
	\begin{enumerate}
		\item For $i\in\set{1,\cdots,l}$, $\phi(L'_i)=L_i$, and $l\geq 3$.
		\item For all $l< i \leq k$, $\phi(L'_i)\neq L_j$ for any $j\in\set{l+1,\cdots,n}$.
	\end{enumerate}
	The gluing $\phi$ is \emph{generic} if:
	\begin{enumerate}
		\item Excepted $L'_1$, $L'_2$ and $L'_3$ no line of $\A'$ is sent by $\phi$ on a line of $\A$ (\emph{i.e.} $l=3$),
		\item No singular point of $\A'$ is sent by $\phi$ on a singular point of $\A$ excepted that $\phi(L'_1\cap L'_2)=L_1\cap L_2$, $\phi(L'_2\cap L'_3)=L_2\cap L_3$ and $\phi(L'_1\cap L'_3)=L_1\cap L_3$.
	\end{enumerate}
	We define also the \emph{glued arrangement}, denoted by $\A \bowtie_\phi \A'$, as the ordered arrangement:
	\begin{equation*}
		\set{L_1,\cdots,L_n,\phi(L'_{l+1}),\cdots,\phi(L'_k)}.
	\end{equation*}
\end{de}

\begin{rmk}\quad
	\begin{enumerate}
		\item  The gluing is not an abelian operator for ordered line arrangements, but it is commutative if we omit the order hypothesis.
		\item Let $\A$ and $\A'$ be two arrangements. It always exists a generic gluing $\phi$ of $\A$ and $\A'$, since the subgroup of automorphism fixing the triangle is of dimension 1 without any fixed point outside the three lines.
	\end{enumerate}
\end{rmk}

There is on $\phi(\A')$ a natural order induced, from the order on $\A'$, by the application $\phi$. With this order, $\phi$ is an homeomorphism preserving both orientation and order between $\A'$ and $\phi(\A')$. Thus we sometimes will use $\A'$ instead of $\phi(\A')$.

\begin{notation}
	The lines of $\A \bowtie_\phi \A'$ are denoted by $D_1,\cdots,D_d$, with $d=n+k-l$, and their meridians by $\mathfrak{m}_1,\cdots,\mathfrak{m}_d$.
\end{notation}

Let $\xi$ (resp. $\xi'$) be a character on $\HH_1(E_\A)$ (resp. $\HH_1(E_{\A'})$) and let $\phi$ be a gluing of $\A$ and $\A'$. We define on $\HH_1(E_{\A\bowtie_\phi\A'})$ the \emph{glued character} $\xi\bowtie_\phi\xi'$ by:
\begin{equation*}
	\xi\bowtie_\phi\xi' :
	\left\{
	\begin{array}{cccl}
		 \HH_1(E_{\A\bowtie_\phi\A'}) & \longrightarrow & \CC^* & \\
		\mathfrak{m}_i & \longmapsto & \xi(m_i)\xi(m'_i), &\text{ for } i\in\set{1,\cdots,l} \\
		\mathfrak{m}_i & \longmapsto & \xi(m_i), & \text{ for } i\in\set{l+1,\cdots,n} \\
		\mathfrak{m}_i & \longmapsto & \xi'(m'_{i-n+l}), & \text{ for } i\in\set{n+1,\cdots,d}
	\end{array}
	\right.,
\end{equation*}
where $m_i\in\HH_1(E_\A)$ (resp. $m'_i\in\HH_1(E_{\A'})$) is the meridian of $L_i$ (resp. $L'_i)$.	If there is no ambiguity, we denote by $\mathfrak{A}_\phi$ the glued arrangement $\A\bowtie_\phi\A'$, and by $\mathfrak{X}_\phi$ the glued character $\xi\bowtie_\phi\xi'$. Let $\mu$ be the cycle of $\HH_1(\Gamma_{\mathfrak{A}_\phi})$ supported by the line $D_1$, $D_2$ and $D_3$.

\begin{propo}\label{Propo_InnerGluing}
	Let $(\A,\xi,\gamma)$ and $(\A',\xi',\gamma')$ be two triangular inner-cyclic arrangements; let $\phi$ be a gluing of $\A$ and $\A'$, then $(\mathfrak{A}_\phi,\mathfrak{X}_\phi,\mu)$ is a triangular inner-cyclic arrangement.
\end{propo}

\begin{proof}
	The cycle $\mu$ is defined by:
\begin{equation*}
	\begin{tikzpicture}[xscale=1.5]
		\node (A1) at (0,0) {$v_{D_1}$};
		\node (A2) at (1,0) {$v_{P_{1,2}}$};
		\node (A3) at (2,0) {$v_{D_2}$};
		\node (A4) at (3,0) {$v_{P_{2,3}}$};
		\node (A5) at (4,0) {$v_{D_{3}}$};
		\node (A6) at (5,0) {$v_{P_{1,3}}$};
		\draw[->] (A1)  -- (A2);
		\draw[->] (A2) -- (A3);
		\draw[->] (A3) -- (A4);
		\draw[->] (A4) -- (A5);
		\draw[->] (A5) -- (A6);
		\draw[->] (A6) to[out=110,in=0] (3.5,0.75) -- (1.5,0.75) to[out=180,in=80] (A1);
	\end{tikzpicture}
\end{equation*}
	To prove that $(\mathfrak{A}_\phi,\mathfrak{X}_\phi,\mu)$ is an inner-cyclic arrangement, we show that the combinatorics of $\mathfrak{A}_\phi$ satisfies the three conditions of Remark~\ref{Rmk_InnerCyclicCombinatorics}. 
	\begin{enumerate}
	
		\item  The cycle $\mu$ is supported by the lines $D_1$, $D_2$ and $D_3$. Their associated meridians $\mathfrak{m}_1$, $\mathfrak{m}_2$ and $\mathfrak{m}_3$ are sent by $\mathfrak{X}_\phi$ on $\xi(m_1)\xi'(m'_1)$, $\xi(m_2)\xi'(m'_2)$ and $\xi(m_3)\xi'(m'_3)$, respectively. Since $(\A,\xi,\gamma)$ and $(\A',\xi',\gamma')$ are triangular inner-cyclic arrangements then $\xi(m_i)=1$ and $\xi'(m'_i)=1$, for $i=1,2,3$. This implies that the three products are sent on $1$ by the character $\mathfrak{X}_\phi$.
		
		\item Let $v_P\in\mu$, then $P=D_i\cap D_j$ with $i\neq j \in\set{1,2,3}$. If $D_q$ is a line of $\mathfrak{A}_\phi$ such that $q\notin\set{1,2,3}$ and $P\in D_q$, then three cases appear: 
			\subitem a) $q\in\set{n+l+1,\cdots,d}$, then $D_q$ comes from a line of $L'_r\in\A'$ (in fact $r=q-n+l$). This line intersects $L'_i \cap L'_j$ and since $(\A',\xi',\gamma')$ is an inner-cyclic arrangement then $\xi'(m'_r)=1$. Finally, $\mathfrak{X}_\phi(\mathfrak{m}_q)=\xi'(m'_r)=1$.
			\subitem b) $q\in\set{l+1,\cdots,n}$, then $D_q$ comes from the line $L_q$ of $\A$. The same arguments as previously work.
			\subitem c) $q\in\set{4,\cdots,l}$, then $D_q$ comes from the line $L'_q\in\A'$ and the line $L_q\in\A$. This implies that $\mathfrak{X}_\phi(\mathfrak{m}_q)=\xi(m_q)\xi'(m'_q)$. But $(\A,\xi,\gamma)$ is an inner-cyclic arrangement, then $\xi(m_q)=1$, since $L_q$ passes through $L_i\cap L_j$. In the same way, $\xi'(m'_q)=1$; and then $\mathfrak{X}_\phi(\mathfrak{m}_q)=\xi(m_q)\cdot\xi'(m'_q)=1$.
		
		\item Let $P\in D_i$, with $i\in\set{1,2,3}$; and let $\D_P$ be the set $\set{D\in\mathfrak{A}_\phi \mid P\in D }$. It can be decomposed in three subsets $\D_P(\A)$, $\D_P(\A')$ and $\D_P(\A,\A')$ composed respectively of the line of $\mathfrak{A}_\phi$ coming from a line in $\A$, a line in $\A'$ and a line in both $\A$ and $\A'$. From this decomposition of  $\D_P$, we obtain:
		\begin{equation*}
			\begin{array}{ccl}
				\prod\limits_{D_i \in \D_P} \mathfrak{X}_\phi(\mathfrak{m}_i) 
				& =  &	\Big( \prod\limits_{D_i \in \D_P(\A)}  \mathfrak{X}_\phi(\mathfrak{m}_i)  \Big) \cdot
											\Big( \prod\limits_{D_i \in \D_P(\A')}  \mathfrak{X}_\phi(\mathfrak{m}_i)  \Big) \cdot
											\Big( \prod\limits_{D_i \in \D_P(\A,\A')}  \mathfrak{X}_\phi(\mathfrak{m}_i) \Big),\\[15pt]
				& =  &	\Big( \prod\limits_{D_i \in \D_P(\A)} \xi(m_i)   \Big) \cdot
											\Big( \prod\limits_{D_i \in \D_P(\A')} \xi'(m'_{i-n+l})  \Big) \cdot
											\Big( \prod\limits_{D_i \in \D_P(\A,\A')} \xi(m_i)\xi'(m'_{i-n+l})  \Big).
			\end{array}
		\end{equation*}
		Since $(\A,\xi,\gamma)$ is an inner-cyclic arrangement, and since $\D_P(\A)$ and $\D_P(\A,\A')$ cover all the indices of the lines of $\A$ passing through $P$, then: 
		\begin{equation*}
			\begin{array}{ccl}
				\Big( \prod\limits_{D_i \in \D_P(\A)} \xi(m_i)  \Big) \cdot \Big(  \prod\limits_{D_i \in \D_P(\A,\A')} \xi(m_i) \Big)
				& = & \prod\limits_{L_i\ni P} \xi(m_i), \\
				& = & 1.
			\end{array}
		\end{equation*}
		In the same manner, 
		$\Big( \prod\limits_{D_i \in \D_P(\A')} \xi'(m'_{i-n+l})  \Big) \cdot \Big(  \prod\limits_{D_i \in \D_P(\A,\A')} \xi'(m'_{i-n+l}) \Big)=1$. 
		Finally, we have:
		\begin{equation*}
			\prod\limits_{D_i \in \D_P} \mathfrak{X}_\phi(\mathfrak{m}_i) = 1.
			\qedhere
		\end{equation*}
	\end{enumerate}
\end{proof}

\begin{thm}\label{Thm_Multiplicativity}
	Let $(\A,\xi,\gamma)$ and $(\A',\xi',\gamma')$ be two triangular inner-cyclic arrangements, and let $\phi$ be a gluing of $\A$ and $\A'$, then:
	\begin{equation*}
		\mathcal{I}(\mathfrak{A}_\phi,\mathfrak{X}_\phi,\mu)=\mathcal{I}(\A,\xi,\gamma)\cdot\mathcal{I}(\A',\xi',\gamma')
	\end{equation*}
	where $\mu\in\HH_1(\Gamma_{\mathfrak{A}_\phi})$ is supported by $D_1$, $D_2$ and $D_3$.
\end{thm}

\begin{proof}
	Let $\tilde{\mu}$ be a nearby cycle associated with $\mu$ in $B_{\mathfrak{A}_\phi}$.  Exceptionally, we denote by $\mathfrak{m}_\ell$, $m_\ell$ and $m'_\ell$ the meridian of $\ell$ in $E_{\mathfrak{A}_\phi}$, $E_\A$ and $E_{\A'}$ respectively (if this makes sense), and also their homology classes.	Recall that $i_*:\HH_1(B_{\mathfrak{A}_\phi})\rightarrow\HH_1(E_{\mathfrak{A}_\phi})$ is the map induced by the inclusion $B_{\mathfrak{A}_\phi} \subset E_{\mathfrak{A}_\phi}$, then we have:
	\begin{equation*}
		i_*([\tilde{\mu}])=\sum\limits_{\ell\in\mathfrak{A}_\phi} \alpha_\ell.\mathfrak{m}_\ell,
	\end{equation*}
	where the $\alpha_\ell$ are integers depending on the choice of the nearby cycle $\tilde{\mu}$. The class of $\tilde{\mu}$ in $\HH_1(E_\A)$ is $\sum\limits_{\ell\in\A} \alpha_\ell.m_\ell$, furthermore $\tilde{\mu}$ is a nearby cycle associated with $\gamma$ in $E_\A$. Similarly, its class in $\HH_1(E_{\A'})$ is $\sum\limits_{\ell\in\A'} \alpha_\ell.m'_\ell$, and $\tilde{\mu}$ is a nearby cycle associated with $\gamma'$ in $E_{\A'}$. From this, we obtain that:
	\begin{align*}
			\mathcal{I}(\mathfrak{A}_\phi,\mathfrak{X}_\phi,\mu)\ 
			& =\  \mathfrak{X}_\phi \circ i_* ([\tilde{\mu}]), \\
			& =\  \mathfrak{X}_\phi\Big(\sum\limits_{\ell\in\mathfrak{A}_\phi} \alpha_\ell.\mathfrak{m}_\ell \Big), \\
			& =\  \prod\limits_{\ell\in\mathfrak{A}_\phi}\mathfrak{X}_\phi (\mathfrak{m}_\ell)^{\alpha_\ell}, \\
			& =\  \prod\limits_{\ell\in\A}\xi(m_\ell)^{\alpha_\ell} \cdot \prod\limits_{\ell\in\A'} \xi'(m'_\ell)^{\alpha_\ell}, \\
			& =\  \mathcal{I}(\A,\xi,\gamma) \cdot \mathcal{I}(\A',\xi',\gamma').
	\qedhere
	\end{align*}
\end{proof}

\section{Extended Rybnikov arrangements and construction of Zariski pairs}\label{Sec_Application}

To illustrate the multiplicativity theorem previously obtained, we prove that extended Rybnikov arrangements form an ordered Zariski pair. Then we extend this method to the family of arrangements with a non real $\mathcal{I}$-invariant.

\subsection{The extended Rybnikov arrangements}

In~\cite{Ryb_Preprint,Ryb_Article}, G.~Rybnikov constructs two arrangements by gluing (along three concurrent lines) two positive MacLane arrangements for the former; and one positive and one negative MacLane arrangements for the latter. Proving that the fundamental groups of these arrangements are not isomorphic, he proves that their topologies are different. In this section, we obtain a similar result, not with the MacLane arrangements, but with the extended MacLane arrangements to deal with inner-cyclic arrangements. 

Let $\phi^+$ be a generic gluing of $\M^+$ and $\M^+$, and let $\phi^-$ be a generic gluing of $\M^+$ and $\M^-$.

\begin{de}
	The extended Rybnikov arrangements are defined by $\R^+=\M^+\bowtie_{\phi^+}\M^+$ and $\R^-=\M^+\bowtie_{\phi^-}\M^-$.
\end{de}

\begin{rmk}
	The extended Rybnikov arrangements are not the Rybnikov arrangements with two additional lines. Indeed, in~\cite{Ryb_Preprint} G.~Rybnikov glues two MacLane arrangements along concurrent lines.
\end{rmk}

\begin{propo}\label{Propo_SameCombinatorics}
	The extended Rybnikov arrangements have the same combinatorics.
\end{propo}

This comes from the fact that the gluings $\phi^+$ and $\phi^-$ considered are generic, and we have $\mathcal{C}_{\R}~=~(\L_{\R},\P_{\R})$, with $\L_{\R}=\set{D_1,\cdots,D_{15}}$ and $\P_{\R}$ is composed of:
\begin{itemize}
	\item the first copy of $\C_\M$:
		\begin{gather*}
			\left\{ 
			\left\{D_1, D_2\right\}, 
			\left\{D_1, D_3\right\}, 
			\left\{D_1, D_4, D_5, D_6\right\}, 
			\left\{D_1, D_7, D_8, D_9\right\}, 
			\left\{D_2, D_3\right\},  
			\right.\\
			\left\{D_2, D_4, D_9\right\}, 
			\left\{D_2, D_5, D_8\right\}, 
			\left\{D_2, D_6, D_7\right\},  
			\left\{D_3, D_4, D_7\right\}, 
			\\
			\left. 
			\left\{D_3, D_5, D_9\right\}, 
			\left\{D_3, D_6, D_8\right\}, 
			\left\{D_4, D_8\right\}, 
			\left\{D_7, D_5\right\}, 
			\left\{D_6, D_9\right\} 
			\right\},
		\end{gather*}
	\item the second copy of $\C_\M$ (without the intersection of $D_1$, $D_2$ and $D_3$):
		\begin{gather*}
				\left\{ 
				\left\{D_1, D_{10}, D_{11}, D_{12}\right\}, 
				\left\{D_1, D_{13}, D_{14}, D_{15}\right\}, 
				\left\{D_2, D_{10}, D_{15}\right\}, 
				\left\{D_2, D_{12}, D_{13}\right\}, 
				\left\{D_3, D_{10}, D_{13}\right\}, 
				\right.\\
				\left.  
				\left\{D_2, D_{11}, D_{14}\right\}, 
				\left\{D_3, D_{11}, D_{15}\right\}, 
				\left\{D_3, D_{12}, D_{14}\right\}, 
				\left\{D_{10}, D_{14}\right\}, 
				\left\{D_{13}, D_{11}\right\}, 
				\left\{D_{12}, D_{15}\right\} 
				\right\},
		\end{gather*}
	
	\item the double points between the two copies of $\C_\M$ (due to the genericity of the gluing):
		\begin{gather*}
			\left\{ \left\{D_4, D_{10}\right\}, \left\{D_4, D_{11}\right\}, \left\{D_4, D_{12}\right\}, \left\{D_4, D_{13}\right\}, \left\{D_4, D_{14}\right\}, \left\{D_4, D_{15}\right\}, \left\{D_5, D_{10}\right\},   \right.\\
			\left. \left\{D_5, D_{11}\right\}, \left\{D_5, D_{12}\right\}, \left\{D_5, D_{13}\right\}, \left\{D_5, D_{14}\right\}, \left\{D_5, D_{15}\right\}, \left\{D_6, D_{10}\right\}, \left\{D_6, D_{11}\right\},    \right.\\
			\left. \left\{D_6, D_{12}\right\}, \left\{D_6, D_{13}\right\}, \left\{D_6, D_{14}\right\}, \left\{D_6, D_{15}\right\}, \left\{D_7, D_{10}\right\}, \left\{D_7, D_{11}\right\}, \left\{D_7, D_{12}\right\},    \right.\\
			\left. \left\{D_7, D_{13}\right\}, \left\{D_7, D_{14}\right\}, \left\{D_7, D_{15}\right\}, \left\{D_8, D_{10}\right\}, \left\{D_8, D_{11}\right\}, \left\{D_8, D_{12}\right\}, \left\{D_8, D_{13}\right\},    \right.\\
			\left. \left\{D_8, D_{14}\right\}, \left\{D_8, D_{15}\right\}, \left\{D_9, D_{10}\right\}, \left\{D_9, D_{11}\right\}, \left\{D_9, D_{12}\right\}, \left\{D_9, D_{13}\right\}, \left\{D_9, D_{14}\right\}, \left\{D_9, D_{15}\right\} \right\}.
		\end{gather*}
\end{itemize}


Let $\phi\in\aut(\C_\R)$ be an automorphism of the extended Rybnikov combinatorics. Since the line $D_1$ is the only one containing four points of multiplicity 4, then it is fixed by $\phi$. In the same way, the lines $D_2$ and $D_3$ are fixed or exchanged by $\phi$. The twelve remaining lines can be combinatorially decomposed in two sets corresponding to the two copies of the extended MacLane combinatorics. Thus we have that $\aut(\C_\R)$ is a subgroup of $\big(\aut(\C_\M)\times\aut(\C_\M) \big) \rtimes \ZZ_2$. Here the $\ZZ_2$ part determines if the two previous sets are fixed or exchanged. Furthermore, the automorphism of the first copy of the extended Maclane, determines only the action on $D_2$ and $D_3$ in the second copy. By similar arguments than in Section~\ref{SubSec_MacLane}, we can represent the automorphism of $\C_\R$ as matrices of $\GL_3(\FF_3)$ of the following types:
\begin{equation*}
	\left(\begin{array}{ccc} 
		\pm 1 & 0 & 0 \\ 
		a & \pm 1 & 0 \\
		b & 0 & \pm 1 
	\end{array}\right)
	\ \text{ or }\ 
	\left(\begin{array}{ccc} 
		\pm 1 & 0 & 0 \\ 
		a & 0 & \pm 1 \\
		b & \pm 1 & 0 
	\end{array}\right),
\end{equation*}
where $a,b$ are in $\FF_3$. Then we can check that $\aut(\C_\R)\simeq \big( (\Sigma_3)^2 \rtimes \ZZ_2 \big) \times \ZZ_2$, where the semi-product by $\ZZ_2$ determines if the two copies are exchanged or not, while the direct-product by $\ZZ_2$ determines if $D_2$ and $D_3$ are exchanged or not.

\subsection{Ordered topology of extended Rybnikov arrangements} 

As an illustration of the Theorem~\ref{Thm_Multiplicativity}, we prove that the extended Rybnikov arrangements form an ordered Zariski pair. After that, we give a way to remove the ordered hypothesis. Even if this result is not a consequence of the result of G.~Rybnikov, it is very close to it. The fact that Rybnikov arrangements satisfy this property is proved in~\cite{Ryb_Preprint,Ryb_Article,ArtCarCogMar_Invariants}, but the techniques used in this paper are new. 

\begin{thm}\label{Thm_Rybnikov}
	There is no ordered-preserving homeomorphism between $(\CC\PP^2,\R^+)$ and $(\CC\PP^2,\R^-)$.
\end{thm}

Before proving this theorem, we have to state the following lemma.

\begin{lem}\label{Lem_OrientationDeletion}
	Let $\A_1$ and $\A_2$ be two arrangements	with the same combinatorics and such that there is no homeomorphism preserving both orientation and order between $(\CC\PP^2,\A_1)$ and $(\CC\PP^2,\A_2)$. If there is no orientation-preserving homeomorphism between $\A_2$ and the complex conjugate of $\A_1$ then there is no order-preserving homeomorphism between $(\CC\PP^2,\A_1)$ and $(\CC\PP^2,\A_2)$.
\end{lem}
It is a consequence of~\cite[Theorem~4.19]{ArtCarCogMar_Topology} (see also~\cite[Theorem 6.4.8]{Gue_Thesis} for a complete proof)

\begin{proof}[Proof of Theorem~\ref{Thm_Rybnikov}]
	We have shown, in Section~\ref{Sec_Invariant}, that both $(\M^+,\xi_\M,\gamma_{(1,2,3)})$ and $(\M^-,\xi_\M,\gamma_{(1,2,3)})$ are triangular inner-cyclic arrangements, where $\xi_\M$ is defined by:
	\begin{equation*}
		\xi_\M : (m_1,\cdots,m_9) \longmapsto (1,1,1,\zeta,\zeta,\zeta,\zeta^2,\zeta^2,\zeta^2),
	\end{equation*}
	with $\zeta$ a 3-root of the unity and $\gamma_{(1,2,3)}$ the cycle supported by the line $L_1$, $L_2$ and $L_3$. Since $\M^+$ and $\M^-$ have the same combinatorics then $\xi_\M$ can be considered as a character on $\HH_1(E_{\M^+})$ or on $\HH_1(E_{\M^-})$. Then, we define, on $\HH_1(E_{\R^+})$ and on $\HH_1(E_{\R^-})$, the same character $\mathfrak{X}$ by $\xi_\M\bowtie_{\phi^+}\xi_\M$ and $\xi_\M\bowtie_{\phi^-}\xi_\M$. Explicitly, we have:
	\begin{equation*}
		\mathfrak{X} : (\mathfrak{m}_1,\cdots,\mathfrak{m}_{15}) \longmapsto (1,1,1,\zeta,\zeta,\zeta,\zeta^2,\zeta^2,\zeta^2,\zeta,\zeta,\zeta,\zeta^2,\zeta^2,\zeta^2).
	\end{equation*}
	
	By Section~\ref{Sec_Invariant} (see also~\cite{ArtFloGue}), we know that $(\M^+,\xi_\M,\gamma_{(1,2,3)})$ and $(\M^-,\xi_\M,\gamma_{(1,2,3)})$ are triangular inner-cyclic arrangements. By Proposition~\ref{Propo_InnerGluing}, we know that $(\R^+,\mathfrak{X},\mu)$ and $(\R^-,\mathfrak{X},\mu)$ are triangular inner-cyclic arrangements, where the cycle $\mu$ of $\Gamma_{\R^+}=\Gamma_{\R^-}$ is supported by $D_1$, $D_2$ and $D_3$. Thus, it makes sense to consider $\mathcal{I}(\R^+,\mathfrak{X},\mu) $ and $\mathcal{I}(\R^-,\mathfrak{X},\mu) $.
	
	By the computations done in Section~\ref{Sec_Invariant}, we have know that: 
	\begin{equation*}
		\mathcal{I}(\M^+,\xi_\M,\gamma_{(1,2,3)})=\zeta^2 \quad\text{ and }\quad \mathcal{I}(\M^-,\xi_\M,\gamma_{(1,2,3)})~=~\zeta.
	\end{equation*}
	Then, by Theorem~\ref{Thm_Multiplicativity}:
	\begin{equation*}
		\mathcal{I}(\R^+,\mathfrak{X},\mu) = \mathcal{I}(\M^+,\xi_\M,\gamma_{(1,2,3)}) \cdot \mathcal{I}(\M^+,\xi_\M,\gamma_{(1,2,3)}) = \zeta^2\cdot\zeta^2=\zeta
	\end{equation*}
	and,
	\begin{equation*}
		\mathcal{I}(\R^-,\mathfrak{X},\mu) = \mathcal{I}(\M^+,\xi_\M,\gamma_{(1,2,3)}) \cdot \mathcal{I}(\M^-,\xi_\M,\gamma_{(1,2,3)}) = \zeta^2\cdot \zeta = 1 .			
	\end{equation*}
	Theorem~\ref{Thm_Invariance} implies that there is no homeomorphism preserving both orientation and order between $(\CC\PP^2,\R^+)$ and $(\CC\PP^2,\R^-)$. Since the $\mathcal{I}$-invariant commutes with the complex conjugacy (see~\cite[Proposition 2.5]{ArtFloGue}), we have that $\mathcal{I}(\overline{\R^+},\mathfrak{X},\gamma) =\overline{\zeta}= \zeta^2$. Then by Theorem~\ref{Thm_Invariance}, there is no homeomorphism preserving orientation and order between $(\CC\PP^2,\overline{\R^+})$ and $(\CC\PP^2,\R^-)$. Applying Lemma~\ref{Lem_OrientationDeletion}, we obtain the result.
\end{proof}

Unfortunately, we cannot obtain a better result on $\R^+$ and $\R^-$ (that is: we cannot remove the ordered hypothesis). Indeed, we have seen that $\aut(\C_\R)\simeq \big( (\Sigma_3)^2 \rtimes \ZZ_2 \big) \times \ZZ_2$. This implies that we can take on the first copy an automorphism with determinant 1, and on the second copy an automorphism with determinant -1. The discussion of Subsection~\ref{SubSec_MacLane} implies that such an automorphism will transform $\R^+$ into $\R^-$. Nevertheless, it is possible to solve this problem by adding some lines in such a way that the automorphisms of each copy are determined by the action on $D_2$ and $D_3$ (for example by adding an additional line in each copies of the extended MacLane passing through a triple point and a double point not in the triangle formed by $D_1$, $D_2$ and $D_3$).  With these two additional lines, we can remove the ordered hypothesis in the previous theorem and then obtain a Zariski pair.

\subsection{Construction of Zarski pairs} 

All the construction previously done for the extended Rybnikov arrangements can be adapted to any inner-cyclic arrangements with a non real $\mathcal{I}$-invariant. Then the following theorem gives a method to construct new examples of Zariski pairs.

If $(\A,\xi,\gamma)$ is a triangular inner-cyclic arrangement, and if $\phi^+$ (resp. $\phi^-$) is a generic gluing of two copies of $\A$ (resp. of a copy of $\A$ and a copy of $\overline{\A}$, the complex conjugate of $\A$), then we denote by $\mathfrak{A}^+$ (resp. $\mathfrak{A}^-$) the glued arrangement associated to $\phi^+$ (resp. $\phi^-$).

\begin{thm}\label{Thm_ZP}
	If $(\A,\xi,\gamma)$ is a triangular inner-cyclic arrangement such that $\mathcal{I}(\A,\xi,\gamma)$ is not real, then there is no order-preserving homeomorphism between $(\CC\PP^2,\mathfrak{A}^+)$ and $(\CC\PP^2,\mathfrak{A}^-)$.
\end{thm}

\begin{proof}
	The proof is similar to Theorem~\ref{Thm_Rybnikov}. Since $\phi^+$ and $\phi^-$ are generic gluings, then $\mathfrak{A}^+$ and $\mathfrak{A}^-$ have the same order combinatorics. Thus, we define on $\HH_1(E_{\mathfrak{A}^+})$ and $\HH_1(E_{\mathfrak{A}^+})$ the same character $\mathfrak{X}$ by the glued characters $\xi\bowtie_{\phi^+}\xi$ and $\xi\bowtie_{\phi^-}\xi$.
	
	By Proposition~\ref{Propo_InnerGluing}, $(\mathfrak{A}^+,\mathfrak{X},\mu)$ and $(\mathfrak{A}^-,\mathfrak{X},\mu)$ are inner-cyclic arrangements, where $\mu$ is supported by $D_1$, $D_2$ and $D_3$. Then Theorem~\ref{Thm_Multiplicativity} implies that:
	\begin{equation*}
		\mathcal{I}(\mathfrak{A}^+,\mathfrak{X},\mu) = \mathcal{I}(\A,\xi,\gamma) \cdot \mathcal{I}(\A,\xi,\gamma) \notin\RR^{*}_{>0},
	\end{equation*}
	and with~\cite[Proposition 2.5]{ArtFloGue},
	\begin{equation*}
		\mathcal{I}(\mathfrak{A}^-,\mathfrak{X},\mu) = \mathcal{I}(\A,\xi,\gamma) \cdot \mathcal{I}(\overline{\A},\xi,\gamma) = 1 .			
	\end{equation*}
	Then, by Theorem~\ref{Thm_Invariance}, there is no homeomorphism preserving both ordered and orientation between $(\CC\PP^2,\mathfrak{A}^+)$ and $(\CC\PP^2,\mathfrak{A}^-)$. We conclude using~\cite[Proposition 2.5]{ArtFloGue} and Lemma~\ref{Lem_OrientationDeletion} as in the proof of Theorem~\ref{Thm_Rybnikov}.
\end{proof}

\begin{cor}
	If the automorphism group of the combinatorics of $\A$ is trivial, then we can remove the hypothesis ``order-preserving'' in Theorem~\ref{Thm_ZP}.
\end{cor}

\begin{proof}
	Let us assume that there is a homeomorphism $\psi$ between $(\CC\PP^2,\mathfrak{A}^+)$ and $(\CC\PP^2,\mathfrak{A}^-)$. Then $\psi$ induces an automorphism $\sigma$ of the combinatorics $\C_\mathfrak{A}$. By hypothesis, $\sigma$ acts trivially on the combinatorics of $\mathfrak{A}$ or exchanges the copies of the combinatorics $\A$. If we change the order on $\mathfrak{A}^+$ by exchanging the order of the two copies of $\A^+$ then $\psi$ becomes an order-preserving homeomorphism. But this change of order is compatible with the arguments of Theorem~\ref{Thm_ZP} proof, which implies a contradiction with the existence of such a homeomorphism. 
\end{proof}

\begin{rmk}
	 The Zariski pairs then obtained are non-arithmetic Zariski pairs (\emph{i.e.} their equations are not conjugated in a field number).
\end{rmk}

\section*{Acknowledgments}

	The results presented here were partially obtained during a visit to Zaragoza University. The author thanks the university for its hospitality. He also would like to thank E.~Artal for all the interesting discussions and his helpful comments on this article.

\bibliographystyle{amsplain}
\bibliography{biblio}

\def\cprime{$'$}
\providecommand{\bysame}{\leavevmode\hbox to3em{\hrulefill}\thinspace}
\providecommand{\MR}{\relax\ifhmode\unskip\space\fi MR }
\providecommand{\MRhref}[2]{%
  \href{http://www.ams.org/mathscinet-getitem?mr=#1}{#2}
}
\providecommand{\href}[2]{#2}
\begin{thebibliography}{10}

\bibitem{Art_Couples}
Enrique Artal, \emph{Sur les couples de {Z}ariski}, J. Algebraic Geom.
  \textbf{3} (1994), no.~2, 223--247. \MR{1257321 (94m:14033)}

\bibitem{ArtCarCogMar_Topology}
Enrique Artal, Jorge Carmona-Ruber, Jos{\'e}~I. Cogolludo-Agust{\'{\i}}n, and
  Miguel~{\'A}. Marco~Buzun{\'a}riz, \emph{Topology and combinatorics of real
  line arrangements}, Compos. Math. \textbf{141} (2005), no.~6, 1578--1588.
  \MR{2188450 (2006k:32055)}

\bibitem{ArtCarCogMar_Invariants}
\bysame, \emph{Invariants of combinatorial line arrangements and {R}ybnikov's
  example}, Singularity theory and its applications (Tokyo), Adv. Stud. Pure
  Math., vol.~43, Math. Soc. Japan, 2006, pp.~1--34. \MR{2313406 (2008g:32042)}

\bibitem{ArtFloGue}
Enrique Artal, Vincent Florens, and Beno{\^i}t Guerville-Ball{\'e}, \emph{A
  topological invariant of line arrangements}, Available at
  \texttt{arXiv:1407.3387}, to appear in Ann. Sc. Norm. Super. Pisa Cl. Sci.,
  2016.

\bibitem{Gue_Thesis}
Beno{\^i}t Guerville-Ball{\'e}, \emph{{Topological invariants of line
  arrangements}}, Ph.D. thesis, Universit{\'e} de Pau et des Pays de l'Adour
  and Universidad de Zaragoza, 2013.

\bibitem{Gue_Pairs}
Beno{\^{\i}}t Guerville-Ball{\'e}, \emph{An arithmetic {Z}ariski 4--tuple of
  twelve lines}, Geom. Topol. \textbf{20} (2016), no.~1, 537--553. \MR{3470721}

\bibitem{Mac}
Saunders MacLane, \emph{Some {I}nterpretations of {A}bstract {L}inear
  {D}ependence in {T}erms of {P}rojective {G}eometry}, Amer. J. Math.
  \textbf{58} (1936), no.~1, 236--240. \MR{1507146}

\bibitem{OrlSol}
Peter Orlik and Louis Solomon, \emph{Combinatorics and topology of complements
  of hyperplanes}, Invent. Math. \textbf{56} (1980), no.~2, 167--189.
  \MR{558866 (81e:32015)}

\bibitem{Ryb_Preprint}
Grigori~L. Rybnikov, \emph{{On the fundamental group of the complement of a
  complex hyperplane arrangement}}, Available at
  \texttt{arXiv:math/9805056v1 [math.AG]}, 1998.

\bibitem{Ryb_Article}
\bysame, \emph{On the fundamental group of the complement of a complex
  hyperplane arrangement}, Funktsional. Anal. i Prilozhen. \textbf{45} (2011),
  no.~2, 71--85. \MR{2848779 (2012i:14067)}

\bibitem{Wes}
Eric Westlund, \emph{{The boundary manifold of an arrangement}}, Ph.D. thesis,
  University of Wisconsin - Madison, 1997.

\bibitem{Whi}
Geoff Whittle, \emph{Oriented matroids (a. bjorner, m. las vergnas, b.
  sturmfels, n. white, and g. ziegler)}, SIAM Review \textbf{36} (1994), no.~4,
  679--680.

\bibitem{Zar_Problem}
Oscar Zariski, \emph{On the {P}roblem of {E}xistence of {A}lgebraic {F}unctions
  of {T}wo {V}ariables {P}ossessing a {G}iven {B}ranch {C}urve}, Amer. J. Math.
  \textbf{51} (1929), no.~2, 305--328. \MR{1506719}

\bibitem{Zar_Topological}
\bysame, \emph{The topological discriminant group of a riemann surface of genus
  p}, American Journal of Mathematics \textbf{59} (1937), no.~2, pp. 335--358
  (English).

\end{thebibliography}

\end{document}